\documentclass[11pt]{article}

\usepackage[T1]{fontenc}
\usepackage[utf8]{inputenc}
\usepackage{amsmath,amssymb,amsthm,mathtools}
\usepackage{libertinus-type1}
\usepackage{libertinust1math}
\usepackage[margin=1in]{geometry}
\usepackage{xcolor}
\usepackage{enumitem}
\usepackage{microtype}
\usepackage[colorlinks=true,linkcolor=blue!50!black,citecolor=blue!50!black,urlcolor=blue!50!black]{hyperref}
\hypersetup{
  pdftitle={Pointed evaluation fibers of rational curves on del Pezzo manifolds},
  pdfauthor={Ari Krishna},
  pdfsubject={Algebraic geometry},
  pdfkeywords={rational curves, del Pezzo manifolds, Kontsevich spaces, evaluation morphisms}
}

\usepackage{aliascnt}
\newtheorem{theorem}{Theorem}[section]
\newaliascnt{proposition}{theorem}
\newtheorem{proposition}[proposition]{Proposition}
\aliascntresetthe{proposition}
\newaliascnt{corollary}{theorem}

\aliascntresetthe{corollary}
\newaliascnt{lemma}{theorem}
\newtheorem{lemma}[lemma]{Lemma}
\aliascntresetthe{lemma}
\theoremstyle{definition}
\newaliascnt{definition}{theorem}

\aliascntresetthe{definition}
\newaliascnt{remark}{theorem}
\newtheorem{remark}[remark]{Remark}
\aliascntresetthe{remark}
\newaliascnt{example}{theorem}

\aliascntresetthe{example}
\newaliascnt{convention}{theorem}

\aliascntresetthe{convention}
\usepackage[nameinlink,noabbrev]{cleveref}
\crefname{theorem}{theorem}{theorems}
\Crefname{theorem}{Theorem}{Theorems}
\crefname{proposition}{proposition}{propositions}
\Crefname{proposition}{Proposition}{Propositions}
\crefname{corollary}{corollary}{corollaries}
\Crefname{corollary}{Corollary}{Corollaries}
\crefname{lemma}{lemma}{lemmas}
\Crefname{lemma}{Lemma}{Lemmas}
\crefname{definition}{definition}{definitions}
\Crefname{definition}{Definition}{Definitions}
\crefname{remark}{remark}{remarks}
\Crefname{remark}{Remark}{Remarks}
\crefname{example}{example}{examples}
\Crefname{example}{Example}{Examples}
\crefname{convention}{convention}{conventions}
\Crefname{convention}{Convention}{Conventions}

\newcommand{\Pj}{\mathbb P}

\newcommand{\OO}{\mathcal O}
\newcommand{\Mbar}{\overline M}
\newcommand{\ev}{\operatorname{ev}}
\newcommand{\Gr}{\operatorname{Gr}}

\newcommand{\Pic}{\operatorname{Pic}}

\title{Pointed evaluation fibers of rational curves on del Pezzo manifolds}

\author{Ari Krishna}

\begin{document}

\maketitle

\begin{abstract}
Let $X$ be a Picard-rank-one del Pezzo manifold of dimension $n\geq 4$ over an algebraically closed field of characteristic zero. Okamura proved that the unpointed Kontsevich spaces $\Mbar_{0,0}(X,d)$ are irreducible of the expected dimension for every $d\geq 1$. We refine this result by studying pointed evaluation fibers. First, we prove that for every $d\geq 1$, the one-pointed evaluation morphism $\Mbar_{0,1}(X,d)\to X$ has geometrically irreducible generic fiber. Second, in the very ample cases $H^n=3,4,5$, we prove that for every $d\geq 2$, the two-pointed evaluation morphism $\Mbar_{0,2}(X,d)\to X\times X$ has geometrically irreducible generic fiber. 
\end{abstract}



\section{Introduction}

Let $X$ be a smooth projective variety and let $H$ be an ample divisor. The Kontsevich space $\Mbar_{0,r}(X,d)$ parametrizes genus-zero stable maps to $X$ of $H$-degree $d$ with $r$ marked points. It is of interest to ask whether the space $\Mbar_{0,0}(X,d)$ is irreducible. Beyond this, a pointed refinement asks for the geometry of the fibers of the evaluation maps
\[
  \ev_{d,r}:\Mbar_{0,r}(X,d)\longrightarrow X^r.
\]
For instance, irreducibility of the generic fiber of $\ev_{d,1}$ means that through a general point of $X$, all degree $d$ rational curves form one irreducible family. Irreducibility of the generic fiber of $\ev_{d,2}$ says that through two general points, all degree $d$ rational curves form one irreducible rational path space. The irreducibility and dimension of these spaces are closely tied to Geometric Manin's Conjecture \cite{LT19}.

In this paper, we consider del Pezzo manifolds of Picard rank one and dimension $n\geq 4$. Thus
\[
  -K_X=(n-1)H,
\]
where $H$ is the ample generator of $\Pic(X)$. Such varieties are classified by the degree $H^n$, with $1\leq H^n\leq 5$ \cite{IP99}. In the very ample cases, one has
\[
\begin{array}{c|c}
H^n & X \\
\hline
3 & \text{a smooth cubic hypersurface in }\Pj^{n+1},\\
4 & \text{a smooth complete intersection of two quadrics in }\Pj^{n+2},\\
5 & \text{a smooth linear section of }\Gr(2,5)\subset \Pj^9.
\end{array}
\]

Okamura proved the following unpointed theorem \cite{Oka24}.

\begin{theorem}[Okamura]\label{thm:okamura}
Let $X$ be a del Pezzo manifold of Picard rank one and dimension $n\geq 4$. Then for every $d\geq 1$, the Kontsevich space $\Mbar_{0,0}(X,d)$ is irreducible of the expected dimension
\[
  (n-1)d+n-3.
\]
\end{theorem}

Building on his work, we prove the following pointed refinements.

\begin{theorem}[One-pointed irreducibility]\label{thm:one-pointed-main}
Let $X$ be a del Pezzo manifold of Picard rank one and dimension $n\geq 4$. Then for every $d\geq 1$, the evaluation morphism
\[
  \ev_{d,1}:\Mbar_{0,1}(X,d)\to X
\]
has geometrically irreducible generic fiber. Consequently, for a general point $p\in X$, the fiber $\ev_{d,1}^{-1}(p)$ is irreducible of dimension
\[
  (n-1)d-2.
\]
\end{theorem}

\begin{theorem}[Two-pointed irreducibility]\label{thm:two-pointed-main}
Let $X$ be a del Pezzo manifold of Picard rank one and dimension $n\geq 4$, and assume $H^n\in\{3,4,5\}$. Then for every $d\geq 2$, the evaluation morphism
\[
  \ev_{d,2}:\Mbar_{0,2}(X,d)\to X\times X
\]
has geometrically irreducible generic fiber. Consequently, for two general points $p,q\in X$, the fiber
\[
  F_{d,2}(p,q):=\ev_{d,2}^{-1}(p,q)
\]
is irreducible of dimension
\[
  (n-1)(d-2)+n-3.
\]
\end{theorem}

The base case is the following computation.

\begin{theorem}[Conics through two general points]\label{thm:conic-base}
Assume $H^n\in\{3,4,5\}$ and let $p,q\in X$ be general. Then $F_{2,2}(p,q)$ is geometrically irreducible of dimension $n-3$. More precisely:
\[
F_{2,2}(p,q)\cong
\begin{cases}
\ev_{1,1}^{-1}(r), & H^n=3,\text{ where } \overline{pq}\cap X=p+q+r,\\
Q^{n-3}, & H^n=4,\\
\Pj^{n-3}, & H^n=5.
\end{cases}
\]
Here, $Q^{n-3}$ denotes a smooth quadric hypersurface in $\Pj^{n-2}$. In each case, the universal conic over $F_{2,2}(p,q)$ is geometrically irreducible.
\end{theorem}

Once \cref{thm:conic-base} is known, \cref{thm:two-pointed-main} follows by the same Stein factorization mechanism as \cref{thm:one-pointed-main}. The test object through two points is a conic joined to a one-pointed tail of degree $(d-2)$. 
The universal conic over $F_{2,2}(p,q)$ is irreducible by \cref{thm:conic-base}, and the tail fibers are irreducible by \cref{thm:one-pointed-main}. Thus, the conic-plus-tail probe has geometrically irreducible generic fiber over $X\times X$, and kills the finite Stein factor of $\ev_{d,2}$.

\section{Preliminaries and Okamura's results}

Throughout, we work over an algebraically closed field $k$ of characteristic zero. All varieties are projective unless otherwise stated. We write $\Mbar_{0,r}(X,d)$ for the coarse Kontsevich moduli space of genus-zero stable maps to $X$ of $H$-degree $d$ with $r$ marked points \cite{FP97}, and we equip it with the reduced structure when necessary; the corresponding statements may equivalently be formulated on the Kontsevich stack.

Let $X$ be a del Pezzo manifold of Picard rank one and dimension $n\geq 4$. Let $H$ be the ample generator of $\Pic(X)$, so
\[
  -K_X=(n-1)H.
\]

We use the following results of Okamura \cite{Oka24}.

\begin{theorem}[Line-fiber theorem]\label{thm:ok-line}
Let
\[
  \ev_{1,1}:\Mbar_{0,1}(X,1)\to X
\]
be the evaluation map for lines. Then a general fiber of $\ev_{1,1}$ is irreducible. Moreover, there is a finite subset $S\subset X$ such that
\[
  \dim \ev_{1,1}^{-1}(p)=n-3\quad \text{for }p\notin S,
\]
and
\[
  \dim \ev_{1,1}^{-1}(p)\leq n-2\quad \text{for }p\in S.
\]
\end{theorem}

\begin{theorem}[Okamura's dimension estimate]\label{thm:ok-dim}
Let
\[
  \ev_{d,1}:\Mbar_{0,1}(X,d)\to X
\]
be the one-pointed evaluation map. With $S\subset X$ as in \cref{thm:ok-line}, for every $d\geq 1$ one has
\[
  \dim \ev_{d,1}^{-1}(p)=(n-1)d-2\quad \text{for }p\notin S,
\]
and
\[
  \dim \ev_{d,1}^{-1}(p)\leq (n-1)d-1\quad \text{for }p\in S.
\]
Furthermore, any component of $\Mbar_{0,0}(X,d)$ generically parametrizes free curves and has expected dimension $(n-1)d+n-3$.
\end{theorem}

\begin{theorem}[Movable bend and break]\label{thm:ok-bb}
Every free curve on $X$ deforms to a chain of free lines.
\end{theorem}

\begin{theorem}[Okamura's unpointed irreducibility theorem]\label{thm:ok-unpointed}
For every $d\geq 1$, the space $\Mbar_{0,0}(X,d)$ is irreducible of dimension
\[
  (n-1)d+n-3.
\]
\end{theorem}

We will also use facts about free curves \cite{Kol96}. Recall that if $f:\Pj^1\to X$ is free, then $f^*T_X$ is globally generated, and therefore,
\[
  H^1(\Pj^1,f^*T_X(-x))=0
\]
for any point $x\in \Pj^1$. This implies smoothness of the one-pointed evaluation map at the corresponding pointed stable map. We employ the deformation theory of free curves and the usual generic-fiber properties of morphisms over a field of characteristic zero. In particular, since $k$ is algebraically closed of characteristic zero, a dominant morphism of $k$-varieties has geometrically irreducible generic fiber if and only if its fiber over a general closed point is irreducible, by openness of geometric irreducibility of fibers; we use this equivalence throughout.

First, let us record the fact that pointed Kontsevich spaces inherit irreducibility from the unpointed space.

\begin{lemma}\label{lem:marked-irreducible}
For every $d\geq 1$ and every $r\geq 0$, the space $\Mbar_{0,r}(X,d)$ is irreducible.
\end{lemma}

\begin{proof}
The case $r=0$ is \cref{thm:ok-unpointed}. The forgetful morphism
\[
  \pi_r:\Mbar_{0,r+1}(X,d)\to \Mbar_{0,r}(X,d)
\]
is the universal curve over the Kontsevich stack, followed on coarse spaces by stabilization. It is flat with connected nodal curve fibers on the stack, and over the dense open locus parametrizing maps from a smooth irreducible source $\Pj^1$, it is simply the universal $\Pj^1$-family. In particular, over a dense open subset of the irreducible base $\Mbar_{0,r}(X,d)$, the total space is irreducible with irreducible generic fiber. Also, flatness prevents an irreducible component of the total space from lying over a proper closed subset of the base. Hence, $\Mbar_{0,r+1}(X,d)$ is irreducible; induction on $r$ proves the lemma.
\end{proof}

\section{Stein factorization lemma}

\begin{lemma}\label{lem:stein-probe}
Let $B$ be a normal integral variety over $k$. Let
\[
  f:M\to B
\]
be a proper dominant morphism with $M$ integral. Let
\[
  M\xrightarrow{\phi}Y\xrightarrow{\pi}B
\]
be the Stein factorization of $f$, so that $\pi$ is finite. Suppose that there exists an integral variety $T$ and a morphism
\[
  \gamma:T\to M
\]
over $B$ such that:
\begin{enumerate}[label=\textup{(\roman*)}]
\item the composite $h=f\circ \gamma:T\to B$ has geometrically irreducible generic fiber;
\item the induced map $\psi:T\to Y$ is dominant.
\end{enumerate}
Then the generic fiber of $f$ is geometrically irreducible.
\end{lemma}

\begin{proof}
Let $K=k(B)$, $E=k(M)$, and $F=k(T)$. Since $B$ is normal, the finite Stein factor $Y$ is the normalization of $B$ in the algebraic closure $L$ of $K$ inside $E$. Thus $k(Y)=L$.

The morphism $\psi:T\to Y$ is dominant, so $L=k(Y)$ embeds into $F=k(T)$ over $K$. Since the generic fiber of $h:T\to B$ is geometrically irreducible, $K$ is algebraically closed in $F$. Thus, the algebraic extension $L/K$ contained in $F/K$ is trivial; it follows that $L=K$.

Thus, $K$ is algebraically closed in $E=k(M)$. Since $\operatorname{char}k=0$, the extension $E/K$ is separably generated. It follows that the generic fiber of $f$ is geometrically integral, hence geometrically irreducible.
\end{proof}

\begin{remark}
Equivalently, one may prove the lemma from the rigidity lemma. After shrinking $B$, the finite morphism $Y\to B$ is finite \'etale and the fibers of $T\to B$ are connected. Since a connected fiber of $T\to B$ maps to the finite discrete fiber of $Y\to B$, the map $T\to Y$ is constant on the fibers of $T\to B$. The rigidity lemma then provides a section $B\to Y$. Dominance of $T\to Y$ implies that this section is dense in $Y$, hence equal to $Y$ after shrinking. Thus, $Y\to B$ has degree one, leading to the same conclusion. 
\end{remark}

\section{One-pointed evaluation fibers}

We now prove \cref{thm:one-pointed-main}.

Let
\[
  L_1:=\Mbar_{0,1}(X,1)_{\mathrm{red}}
\]
be the one-pointed line space and let
\[
  e:L_1\to X
\]
be the line evaluation map. By \cref{thm:ok-line} and openness of geometric irreducibility of fibers, there is a nonempty open subset
$U\subset X$
such that
$e^{-1}(U)\to U$
is flat with geometrically irreducible fibers.

Let
\[
  L_2:=\Mbar_{0,2}(X,1)_{\mathrm{red}}
\]
and define
\[
  P:=(\ev_1,\ev_2)^{-1}(U\times U)\subset L_2.
\]
A point of $P$ is a line together with two marked points on it, both mapping to $U$. Let
\[
  a,b:P\to U
\]
be the two evaluation maps.

\begin{lemma}\label{lem:P-fibers}
After possibly shrinking $U$, the morphism
\[
  a:P\to U
\]
is flat with geometrically irreducible fibers.
\end{lemma}

\begin{proof}
For a point $x\in U$, the fiber $a^{-1}(x)$ is an open subset of the universal curve over the line fiber $e^{-1}(x)$. The line fiber is geometrically irreducible by the choice of $U$, and the universal curve over it is generically a $\Pj^1$-family, hence geometrically irreducible. Removing the condition that the second marked point land outside $U$ removes a closed subset not equal to the whole fiber, because every line through $x\in U$ meets $U$. Thus, $a^{-1}(x)$ is geometrically irreducible. Flatness follows after shrinking $U$ by generic flatness.
\end{proof}

For $d\geq 1$, define the chain space as the iterated fiber product
\[
  \Theta_d=P\times_{U}P\times_{U}\cdots\times_{U}P
\]
with $d$ factors. A point of $\Theta_d$ is a chain
\[
  x_0\xrightarrow{\ell_1}x_1\xrightarrow{\ell_2}\cdots\xrightarrow{\ell_d}x_d,
\]
where all $x_i\in U$ and $x_{i-1},x_i\in \ell_i$. Let
\[
  \alpha_d:\Theta_d\to U
\]
be the morphism that remembers the starting point $x_0$.

\begin{lemma}\label{lem:theta-fibers}
For every $d\geq 1$, the morphism
\[
  \alpha_d:\Theta_d\to U
\]
is flat with geometrically irreducible fibers.
\end{lemma}

\begin{proof}
For $d=1$, this is \cref{lem:P-fibers}. We argue by induction on $d$. Assume the result holds for $d-1$. Then,
\[
  \Theta_d=\Theta_{d-1}\times_{\beta_{d-1},U,a}P,
\]
where $\beta_{d-1}$ records the endpoint of the chain of length $d-1$. The projection $\Theta_d\to \Theta_{d-1}$ is the base change of $a:P\to U$, so it is flat with geometrically irreducible fibers. Since the fibers of $\alpha_{d-1}$ are geometrically irreducible by induction, the fibers of $\alpha_d$ are geometrically irreducible. Flatness follows by composition.
\end{proof}

There is a gluing morphism
\[
  \gamma_d:\Theta_d\to \Mbar_{0,1}(X,d)_{\mathrm{red}}
\]
which glues the chain of $d$ lines and forgets every marking except the first. It satisfies
\[
  \ev_{d,1}\circ \gamma_d=\alpha_d.
\]

\begin{proof}[Proof of \cref{thm:one-pointed-main}]
Set
\[
  M_{d,1}:=\Mbar_{0,1}(X,d)_{\mathrm{red}}.
\]
By \cref{lem:marked-irreducible}, $M_{d,1}$ is integral. Consider the Stein factorization
\[
  M_{d,1}\to Y_d\to X
\]
of $\ev_{d,1}$. The map $\Theta_d\to X$ dominates the open set $U\subset X$, hence dominates $X$. Therefore, the image of $\Theta_d\to Y_d$ maps dominantly to $X$. Since $Y_d\to X$ is finite and $Y_d$ is irreducible, this image is dense in $Y_d$.

By \cref{lem:theta-fibers}, the map $\Theta_d\to X$ has geometrically irreducible generic fiber. Endowed with its reduced structure, $\Theta_d$ is integral: it is irreducible by \cref{lem:theta-fibers}, being flat over the irreducible base $U$ with geometrically irreducible fibers, and passing to the reduction affects neither this property nor the geometric irreducibility of the generic fiber. Thus, the hypotheses of \cref{lem:stein-probe} are satisfied with $T=\Theta_d$, and $\ev_{d,1}$ has geometrically irreducible generic fiber.

The dimension statement follows from \cref{thm:ok-dim}.
\end{proof}

\section{Conics through two general points}

We now prove \cref{thm:conic-base}. Throughout this section, $p,q\in X$ are general points.

\begin{lemma}\label{lem:no-line-through-two-points}
For two general points $p,q\in X$, there is no line on $X$ containing both $p$ and $q$.
\end{lemma}

\begin{proof}
The space of lines has dimension $2n-4$ by \cref{thm:ok-line}. The two-pointed universal family of lines has dimension $2n-2$, while $X\times X$ has dimension $2n$. Hence, its image in $X\times X$ is a proper closed subset. A general pair $(p,q)$ is not in this image.
\end{proof}

We will use the following reduction from stable maps to conics.

\begin{lemma}\label{lem:degree-two-incidence}
Let $p,q\in X$ be a pair of points which are not contained in a line on $X$. Then no point of $F_{2,2}(p,q)$ is a double cover of a line. Moreover, after passing to reduced structures, $F_{2,2}(p,q)$ is naturally identified with the incidence space of connected degree-two genus-zero curves $C\subset X$ containing $p$ and $q$. Such a curve is either an integral conic or a chain of two lines, with $p$ and $q$ lying on the two distinguished ends in the reducible case.
\end{lemma}

\begin{proof}
A degree-two stable map whose image is a line is a double cover of that line. Since the image contains both marked points, this would give a line on $X$ containing $p$ and $q$, contrary to the hypothesis. Thus every nonconstant image has degree two.

If the source is irreducible, the map is birational onto a degree-two curve, hence onto an integral conic. The normalization of an integral conic, together with the two inverse images of $p$ and $q$, recovers the stable map uniquely. If the source is reducible, the two nonconstant components have degree one. Since no line on $X$ contains both $p$ and $q$, the two markings lie on different line components, and the image is a chain of two lines joining $p$ to $q$. Conversely, such a chain determines a unique two-pointed stable map by mapping each component isomorphically to the corresponding line.

No contracted component can occur in these two-pointed degree-two maps. Indeed, there are at most two nonconstant components. A contracted component carrying no marking would need at least three nodes, which is impossible. If a contracted component carries one marking, say the marking mapping to $p$, then stability forces it to meet at least two nonconstant components; the other marking mapping to $q$ lies on one of the corresponding degree-one components, forcing a line on $X$ to contain both $p$ and $q$. This is excluded. A contracted component cannot carry both markings because $p\neq q$. These constructions are algebraic in families, so they identify the reduced Kontsevich fiber with the corresponding conic-incidence space.
\end{proof}

\section{Cubic hypersurfaces}

Assume $H^n=3$, so $X=X_3\subset \Pj^{n+1}$ is a smooth cubic hypersurface. Rational curves on cubic hypersurfaces, and more generally on hypersurfaces of low degree, have been studied in \cite{CS09,HRS04}.

\begin{proposition}\label{prop:cubic-conics}
For general $p,q\in X$, the fiber $F_{2,2}(p,q)$ is geometrically irreducible of dimension $n-3$. More precisely, if
\[
  L:=\overline{pq}
\]
and
\[
  L\cap X=p+q+r
\]
scheme-theoretically, then
\[
  F_{2,2}(p,q)\cong \ev_{1,1}^{-1}(r).
\]
\end{proposition}

\begin{proof}
For general $p,q$, the line $L=\overline{pq}$ is not contained in $X$, so $L\cap X$ is a length-three scheme. Write
\[
  L\cap X=p+q+r.
\]
The residual point $r$ is general as $(p,q)$ varies: indeed, the incidence of ordered triples of collinear points on $X$ is symmetric in the three points and dominates each factor.

Let $C\subset X$ be a conic through $p$ and $q$, and let $\Pi=\langle C\rangle\simeq \Pj^2$. Since $p,q\in C$, the plane $\Pi$ contains $L$. As $L\not\subset X$ and $L\subset \Pi$, the plane $\Pi$ is not contained in $X$; hence $\Pi\cap X$ is a plane cubic containing $C$, so
\[
  \Pi\cap X=C+\ell
\]
for a residual line $\ell\subset \Pi$. Because $\Pi\cap X\subset X$, the residual line lies on $X$. Intersecting with $L$, the conic $C$ accounts for the points $p$ and $q$, so the residual line accounts for the residual point $r$. Thus, $r\in \ell$.

Conversely, let $\ell\subset X$ be a line through $r$. Since $L$ is not contained in $X$, $\ell\neq L$. Let
\[
  \Pi=\langle L,\ell\rangle.
\]
Then, $\Pi\cap X$ is a plane cubic containing $\ell$, so
\[
  \Pi\cap X=\ell+C
\]
for a residual conic $C$. Since $L\cap X=p+q+r$ and $\ell$ contains $r$, the residual conic contains $p$ and $q$. These two constructions are inverse to one another. Hence,
\[
  F_{2,2}(p,q)\cong \ev_{1,1}^{-1}(r).
\]
By Okamura's line-fiber theorem, the latter is geometrically irreducible of dimension $n-3$ for general $r$.
\end{proof}

\section{Intersections of two quadrics}

Assume $H^n=4$, so
\[
  X=Q_1\cap Q_2\subset \Pj^{n+2}
\]
is a smooth complete intersection of two quadrics.

\begin{proposition}\label{prop:two-quadric-conics}
For general $p,q\in X$, the fiber $F_{2,2}(p,q)$ is isomorphic to a smooth quadric $Q^{n-3}\subset \Pj^{n-2}$. In particular, it is geometrically irreducible of dimension $n-3$.
\end{proposition}

\begin{proof}
Let $L=\overline{pq}$. The restrictions $Q_1|_L$ and $Q_2|_L$ are binary quadrics vanishing at $p$ and $q$, hence they are proportional. Thus there is a unique member
\[
  Q_0\in \langle Q_1,Q_2\rangle
\]
of the pencil containing $L$. For general $p,q$, this quadric $Q_0$ is smooth: the pencil has only finitely many singular members, and for each fixed member the condition that the chord line $\overline{pq}$ be contained in it is a proper closed condition on $X\times X$.

Let $C\subset X$ be a conic through $p$ and $q$, and let $\Pi=\langle C\rangle$. Then $L\subset \Pi$. Since $C\subset Q_1\cap Q_2$, the restrictions $Q_1|_\Pi$ and $Q_2|_\Pi$ are plane quadrics vanishing on the same conic. Hence they are linearly dependent. Therefore, some member of the pencil $\langle Q_1,Q_2\rangle$ vanishes identically on $\Pi$. Since $\Pi$ contains $L$, that member is $Q_0$. Thus, $\Pi\subset Q_0$.

Conversely, suppose $\Pi$ is a plane with
\[
  L\subset \Pi\subset Q_0.
\]
Let $Q_\infty$ be any member of the pencil independent from $Q_0$. Since $X=Q_0\cap Q_\infty$, the intersection
\[
  \Pi\cap X=\Pi\cap Q_\infty
\]
is a plane conic containing $p$ and $q$, for general $(p,q)$; no plane contained in $X$ contains a general chord $L$. Thus
\[
  F_{2,2}(p,q)\cong \{\Pi: L\subset \Pi\subset Q_0\}.
\]

We now identify this parameter space. Let $E$ be the vector space with $\Pj(E)=\Pj^{n+2}$, and let $W\subset E$ be the two-dimensional totally isotropic subspace corresponding to the line $L\subset Q_0$. Planes $\Pi$ satisfying $L\subset \Pi\subset Q_0$ correspond to three-dimensional totally isotropic subspaces $U\subset E$ with $W\subset U$. Equivalently, they correspond to isotropic lines in the nondegenerate quadratic space
\[
  W^\perp/W.
\]
This vector space has dimension
\[
  (n+3)-2-2=n-1.
\]
Therefore, the parameter space is the smooth quadric of isotropic lines in $\Pj(W^\perp/W)\simeq \Pj^{n-2}$, namely, a smooth $Q^{n-3}$.
\end{proof}

\section{Linear sections of the Grassmannian}

Assume $H^n=5$, so
\[
  X=\Gr(2,V)\cap \Lambda,
\]
where $\dim V=5$ and $\Lambda\simeq \Pj^{n+3}\subset \Pj(\wedge^2V)\simeq \Pj^9$. Schubert varieties in such linear sections of the Grassmannian are studied in \cite{AC12}. Let
\[
  p=[A],\qquad q=[B]
\]
with $A,B\subset V$ two-dimensional subspaces. For general $p,q$, we have $A\cap B=0$. Set
\[
  W:=A\oplus B.
\]
Then, $\dim W=4$, and
\[
  \Gr(2,W)\subset \Gr(2,V)
\]
is a smooth quadric fourfold in $\Pj(\wedge^2W)\simeq \Pj^5$.

\begin{lemma}\label{lem:grass-conics-in-W}
Every degree-two stable map to $\Gr(2,V)$ passing through $[A]$ and $[B]$, with $A\cap B=0$, has image contained in $\Gr(2,W)$.
\end{lemma}

\begin{proof}
First, consider an irreducible degree-two map $f:\Pj^1\to \Gr(2,V)$. It corresponds to a rank-two subbundle
\[
  S\subset V\otimes \OO_{\Pj^1}
\]
with $\deg S=-2$. The splitting type of $S$ is either $\OO(-1)\oplus \OO(-1)$ or $\OO\oplus \OO(-2)$. In the second case all two-planes in the family contain a fixed line, so the image cannot pass through two disjoint two-planes $A$ and $B$. In the balanced case $S\simeq \OO(-1)\oplus \OO(-1)$, the subspace of $V$ generated by the fibers of $S$ has dimension at most four. Since two fibers are $A$ and $B$, whose sum has dimension four, that subspace must be $W=A\oplus B$. Thus, the image lies in $\Gr(2,W)$.

For a reducible degree-two map, the image is a chain of two Grassmannian lines. If an intermediate two-plane $C$ lies on a line with $A$ and on a line with $B$, then $C$ meets both $A$ and $B$ in lines. Since $A\cap B=0$, those two lines span $C$, so $C\subset A\oplus B$. Hence, the reducible image also lies in $\Gr(2,W)$.
\end{proof}

\begin{proposition}\label{prop:grass-conics}
For general $p,q\in X$,
\[
  F_{2,2}(p,q)\cong \Pj^{n-3}.
\]
In particular $F_{2,2}(p,q)$ is irreducible of dimension $n-3$.
\end{proposition}

\begin{proof}
By \cref{lem:grass-conics-in-W}, conics through $p,q$ lie in
\[
  Q^4:=\Gr(2,W)\subset \Pj(\wedge^2W)\simeq \Pj^5.
\]
Let $L:=\overline{pq}$. Since $A\cap B=0$, the points $[A]$ and $[B]$ are not joined by a line lying on the Klein quadric $Q^4=\Gr(2,W)$, so $L\not\subset Q^4$; consequently no plane containing $L$ is contained in $Q^4$. Hence, for every plane $\Pi$ with $L\subset \Pi\subset \Pj(\wedge^2W)$, the intersection
\[
  C_\Pi=\Pi\cap Q^4
\]
is a conic through $p,q$, and conversely every conic in $Q^4$ through $p,q$ arises in this way with $\Pi=\langle C_\Pi\rangle$.

The conic $C_\Pi$ lies in $X=\Gr(2,V)\cap \Lambda$ if and only if its span $\Pi$ lies in $\Lambda$. Hence, conics in $X$ through $p,q$ are parametrized by planes
\[
  \Pi \quad \text{with}\quad L\subset \Pi\subset \Lambda\cap \Pj(\wedge^2W).
\]
It remains to compute the dimension of this linear space. Let
\[
  K:=\Lambda^\perp\subset \wedge^2 V^\vee
\]
be the annihilator of $\Lambda$, so $\dim K=6-n$. The projective dual of $\Gr(2,V)$ is the Pfaffian locus of skew forms of rank at most two. Since $X=\Gr(2,V)\cap \Lambda$ is smooth, $\Pj(K)$ avoids this dual variety. Hence every nonzero form in $K$ has rank four. A rank-four skew form on the five-dimensional vector space $V$ has maximal isotropic subspaces of dimension three, so its restriction to the four-dimensional subspace $W$ is nonzero. Therefore the restriction map
\[
  K\longrightarrow \wedge^2 W^\vee
\]
is injective. It follows that
\[
  \Lambda_W:=\Lambda\cap \Pj(\wedge^2W)
\]
has codimension $6-n$ in $\Pj(\wedge^2W)\simeq \Pj^5$, and hence
\[
  \dim \Lambda_W=5-(6-n)=n-1.
\]
Thus, $\Lambda_W\simeq \Pj^{n-1}$. The planes containing a fixed line $L$ inside $\Pj^{n-1}$ are parametrized by $\Pj^{n-3}$. This proves the proposition.
\end{proof}

\begin{proof}[Proof of \cref{thm:conic-base}]
The three descriptions of the conic fiber are \cref{prop:cubic-conics}, \cref{prop:two-quadric-conics}, and \cref{prop:grass-conics}. In each description, the reducible conics form a proper closed discriminant locus in the displayed geometrically irreducible parameter space. Hence the generic conic is integral. The universal conic over $F_{2,2}(p,q)$ therefore has geometrically integral generic fiber over a geometrically irreducible base, and so is geometrically irreducible.
\end{proof}

\section{Two-pointed evaluation fibers}

We now prove \cref{thm:two-pointed-main}. Let $X$ be a Picard-rank-one del Pezzo manifold of dimension $n\geq 4$ with $H^n\in\{3,4,5\}$.

For $d=2$, the theorem follows from \cref{thm:conic-base}. Assume, henceforth, that $d\geq 3$, and set
\[
  e:=d-2.
\]
Let
\[
  M_{d,2}:=\Mbar_{0,2}(X,d)_{\mathrm{red}}.
\]
By \cref{lem:marked-irreducible}, $M_{d,2}$ is integral.

We construct the conic-plus-tail probe over a dense open subset of $X\times X$. Let
\[
  g: \Mbar_{0,3}(X,2)_{\mathrm{red}}\longrightarrow X\times X
\]
be the morphism given by the first two evaluation maps, and let
\[
  \ev_3:\Mbar_{0,3}(X,2)_{\mathrm{red}}\to X
\]
be the third evaluation map. By \cref{lem:marked-irreducible}, the source of $g$ is integral. By \cref{thm:conic-base,lem:degree-two-incidence}, the general fiber of $g$ is the universal conic over a general conic fiber, and is geometrically irreducible. After replacing $X\times X$ by a dense open subset $B^\circ$, we may therefore assume that
\[
  g^{-1}(B^\circ)\to B^\circ
\]
is flat with geometrically irreducible fibers.

By \cref{thm:one-pointed-main} and generic flatness, there is a dense open subset $U_e\subset X$ such that
\[
  M_e^\circ:=\ev_{e,1}^{-1}(U_e)\subset \Mbar_{0,1}(X,e)_{\mathrm{red}}
\]
is flat over $U_e$ with geometrically irreducible fibers. Shrinking $B^\circ$ further, we may assume $B^\circ\subset U_e\times U_e$. Define
\[
  C_2^\circ:=g^{-1}(B^\circ)\cap \ev_3^{-1}(U_e).
\]
Then $C_2^\circ\to B^\circ$ is flat with geometrically irreducible fibers: over a point $(p,q)\in B^\circ$, the fiber is the inverse image of $U_e$ in the universal conic over $F_{2,2}(p,q)$, which is a nonempty open subset because $p,q\in U_e$. In particular, $C_2^\circ$ is integral.

Now set
\[
  T_d:=C_2^\circ\times_{U_e} M_e^\circ,
\]
where the fiber product is taken using the third evaluation map on $C_2^\circ$ and the marked evaluation map on $M_e^\circ$. The variety $T_d$ is integral: the projection $T_d\to C_2^\circ$ is the base change of $M_e^\circ\to U_e$, hence is flat with geometrically irreducible fibers over the integral variety $C_2^\circ$. A point of $T_d$ is a conic with marked points $p,q,x$, with $x\in U_e$, together with a degree $e$ tail through $x$. Gluing at $x$ furnishes a morphism
\[
  \gamma_d:T_d\to M_{d,2}.
\]
There is also a map
\[
  h_d:T_d\to X\times X
\]
remembering the first two marked points, and
\[
  \ev_{d,2}\circ \gamma_d=h_d.
\]

\begin{lemma}\label{lem:Td-generic-fiber}
The morphism
\[
  h_d:T_d\to X\times X
\]
has geometrically irreducible generic fiber.
\end{lemma}

\begin{proof}
It suffices to work over the dense open subset $B^\circ\subset X\times X$. Fix $(p,q)\in B^\circ$. The fiber of $h_d$ over $(p,q)$ is
\[
  T_d(p,q)=C_2^\circ(p,q)\times_{U_e} M_e^\circ,
\]
where $C_2^\circ(p,q)$ is a nonempty open subset of the universal conic over $F_{2,2}(p,q)$. By the construction of $B^\circ$, the variety $C_2^\circ(p,q)$ is geometrically irreducible. The projection
\[
  T_d(p,q)\to C_2^\circ(p,q)
\]
is the base change of $M_e^\circ\to U_e$, and hence is flat with geometrically irreducible fibers. Therefore $T_d(p,q)$ is geometrically irreducible. Thus $h_d$ has geometrically irreducible generic fiber.
\end{proof}

\begin{proof}[Proof of \cref{thm:two-pointed-main}]
Consider the Stein factorization
\[
  M_{d,2}\to Y_d\to X\times X
\]
of $\ev_{d,2}$. The image of $T_d\to X\times X$ is dense, and $Y_d\to X\times X$ is finite. Since $Y_d$ is irreducible, the induced map
\[
  T_d\to Y_d
\]
is dominant.

By \cref{lem:Td-generic-fiber}, $T_d\to X\times X$ has geometrically irreducible generic fiber. Applying \cref{lem:stein-probe} with $B=X\times X$, $M=M_{d,2}$, and $T=T_d$, we conclude that $\ev_{d,2}$ has geometrically irreducible generic fiber.

Finally, the dimension follows from Okamura's dimension formula for the unpointed space. We have
\[
  \dim \Mbar_{0,0}(X,d)=(n-1)d+n-3,
\]
so
\[
  \dim \Mbar_{0,2}(X,d)=(n-1)d+n-1.
\]
Since $\dim(X\times X)=2n$ and $\ev_{d,2}$ is dominant, the general fiber has dimension
\[
  (n-1)d+n-1-2n=(n-1)(d-2)+n-3.
\]
This completes the proof.
\end{proof}

\section*{Acknowledgments}

The author thanks Professor Joe Harris for his guidance.

\bigskip

\noindent\textsc{Department of Mathematics, Harvard University, 1 Oxford Street, Cambridge, MA 02138, USA}\par
\noindent\textit{Email address:} \texttt{akrishna@college.harvard.edu}

\end{document}